\documentclass[12pt]{amsart}
\usepackage{amsmath,amssymb,amsthm}
\newtheorem{thm}{Theorem}

\theoremstyle{definition}

\newtheorem{lemma}{Lemma}
\theoremstyle{remark}
\newtheorem{rk}{Remark}

\begin{document}

\title[Liouville theorem with parameters]{Liouville theorem with parameters:\\ asymptotics of certain rational integrals in differential fields}

\author{Ma{\l}gorzata Stawiska}
\address{Department of Mathematics, University of Kansas, Lawrence, KS 66045}
\email{stawiska@ku.edu}
\thanks{}
\dedicatory{} 
 \maketitle
\begin{abstract}

We study asymptotics of integrals of certain rational functions that depend on parameters in a field $K$ of characteristic zero. We use formal power series to represent the integral and prove certain identities about its coefficients following from generalized Vandermonde determinant expansion. 
Our result can be viewed as a parametric version of a classical theorem of Liouville. We also give applications.

\end{abstract}

\section{Integrating rational functions over differential fields} \label{s:functions}

Let $\mathcal{D}$ be a differential field. This means that $\mathcal{D}$ is a field  with an additional mapping $': \mathcal{D} \mapsto \mathcal{D}$ (differentiation) satisfying the following two conditions: $(u+v)' = u'+v'$ and $(uv)'=u'v+uv'$ for all $u,v \in \mathcal{D}$. The set $K=\{u \in \mathcal{D}: u'=0\}$ is a subfield of $\mathcal{D}$, called the field of constants.  We use the following terminology, adapted from \cite{Ris}:  Let $U$ be a universal (differential) extension of $\mathcal{D}$.  For $u \in U$, $u$ and $\mathcal{D}(u)$  are said to be simple elementary over   $\mathcal{D}$ iff one of the
following conditions holds:
   (1) $u$ is algebraic over  $\mathcal{D}$;
   (2) There is a $v$ in  $\mathcal{D}$, $v \neq 0$ such that $v'=vu'$ (we will write equivalently $u=\log v$); 
   (3) There is a  $v$ in  $\mathcal{D}$, $v \neq 0$ such that $u'=uv'$. (or equivalently $u = \exp v$).\\

We say that $\mathcal{F}$ and any $w \in \mathcal{F}$ is elementary over $\mathcal{D}$ if  $\mathcal{F}=\mathcal{D}(u_1,...,u_n)$ for some $n$, where each $u_i$ is simple elementary over $\mathcal{D}(u_1,...,u_{i-1}), \quad i=1,...,n$.\\

The following theorem dates back to J. Liouville (cf. \cite{Ris}):\\

\begin{thm}  Let $\mathcal{D}$ be a differential field, $\mathcal{F}$
 elementary over $\mathcal{D}$. Suppose $\mathcal{D}$ and $\mathcal{F}$ have the same constant field K. Let $g \in  \mathcal{F}$, $f \in \mathcal{D}$
 with $g'=f$. Then $g= v_0 + \sum c_i \log v_i$, where $v_0,v_i$ are elements of $\mathcal{D}$ and $c_i$ are elements of K.
\end{thm}

We will write $f=g$ as equivalent to $g'=f$.\\

\textbf{Example  and notation:} Let $K$ be an arbitrary field of characteristic zero and  $z$ be transcendental over $K$. We introduce differentiation in the polynomial ring $K[z]$ by taking $z'=1$ and $a'=0$ for all $a \in K$ (the standard differentiation of polynomials in one variable). The field $K(z)$ of rational fractions of $K[z]$ is a differential field when we extend $(z^{-1})'=-z^{-1}z^{-1}$, and $K$ is its field of constants.\\

With $K$ as in the example, we will consider the ring $K[[1/z]$ 
of the following formal series:
  $\sum_{n=0}^{\infty} a_n z^{-n}$ with $a_n \in K$. The differentiation ' can be extended term-by-term as a map of $K[[1/z]]$ to itself. One can also define a valuation $o: K[[1/z]]\mapsto \mathbb{N}\cup \{\infty\}$ as follows (cf. \cite{VS}, discussion before Proposition 2.3.16):                                                                            $o(f)=\min\{n: a_n \neq 0\}$ and $o(0)=\infty$. \\
Consider a  a square-free polynomial $Q(z)=z(z-a_1)...(z-a_q)$  with $a_n \in K$.  Our result can be now formulated as follows:\\

  \begin{thm}\label{thm:main} (a) In an elementary field $\mathcal{F}$ over $K(z)$, consider the elements $g=\int f \in \mathcal{F}$ with $f=1/Q$, where $Q(z)=z(z-a_1)...(z-a_q)$ is a square-free polynomial with $a_1,...,a_q \in K, \quad q \geq 1$. The set $G$ of all such elements is  in a bijective correspondence with a subset of $K[[1/z]]$.  \\
(b)  For (the image of) a $g =\int 1/Q$ we have $o(g)=q$, where $q=\deg Q+1$. 
 \end{thm}

In the proof  of this theorem we will apply the following identities:\\
\begin{lemma}
\[
\frac{1}{Q'(0)}+\frac{1}{Q'(a_1)}+...+\frac{1}{Q'(a_q)}=0,
\]
\[
\frac{a_1}{Q'(a_1)}+...+\frac{a_q}{Q'(a_q)}=0,
\]
...\\
\[
\frac{a_1^{q-1}}{Q'(a_1)}+...+\frac{a_q^{q-1}}{Q'(a_q)}=0,
\]
\[
\frac{a_1^{q}}{Q'(a_1)}+...+\frac{a_q^{q}}{Q'(a_q)}=1,
\]
\[
\frac{a_1^{q+l}}{Q'(a_1)}+...+\frac{a_q^{q+l}}{Q'(a_q)}=S_l(a_1,...,a_q),
\]
where $S_l$ is the complete homogeneous polynomial of
degree $l$, symmetric in its variables, i.e., $S_l(X_1,...,X_n)=\sum_{1\leq i_1\leq...\leq i_l\leq n }X_{i_1}...X_{i_l}$ for $l=1,2,...$.
\end{lemma}
\begin{proof} (of Lemma) We use properties of the Vandermonde
determinant:\\
\[
V_n(x_1,...,x_n)=\vmatrix
1 & x_1 &\hdots &x_1^{n-1}\\
1 & x_2 &\hdots &x_2^{n-1}\\
\hdotsfor4\\
1 & x_n &\hdots &x_n^{n-1}\\
\endvmatrix
\]
Recall that $V_n(x_1,...,x_n)= \prod_{1 \leq i<j \leq n}(x_j-x_i)$
and
$V_{n+1}(x_1,...,x_{n+1})=(-1)^n\prod_{i=1}^n(x_i-x_{n+1})V_n(x_1,...,x_n)$.
More generally, one can consider
\[ V_{n,l}(x_1,...,x_n)=\vmatrix
1 & x_1 &\hdots &x_1^{n-1+l}\\
1 & x_2 &\hdots &x_2^{n-1+l}\\
\hdotsfor4\\
1 & x_n &\hdots &x_n^{n-1+l},\\
\endvmatrix
\] where $l=1,2,...$. Then $V_{n,l}(x_1,...,x_n)=V_n(x_1,...,x_n)\cdot
S_l(x_1,...,x_n)$, where $S_l$ is the complete homogeneous polynomial of
degree $l$ in $x_1,...,x_n$ (\cite{Ma}, formula I.3.1).\\
Note that 
for  $Q(z)=z(z-a_1)...(z-a_q)$ one has
$Q'(0)=(-1)^q\prod_{i=1}^q a_1...a_q$,
$Q'(a_i)=\prod_{j \neq i}(a_i-a_j)$.
To prove the first stated identity, let us make
$V_{q+1}(0,a_1,...,a_q)$ the common denominator of the
left hand side. Then
$1/Q'(0)=(-1)^qV_q(a_1,...,a_q)/V_{q+1}(0,a_1,...,a_q)$ and
$1/Q'(a_i)=(-1)^{n-i-1}\prod_{k \neq i}a_k \prod_{k,j
\neq 1, k<j}(a_k-a_j)$, so the numerator is the cofactor
of the element 1 in the $i$-th row of $V_q$. Thus in the summation
of all terms corresponding to different roots of $Q$ the
numerator is $V_q$ minus its Laplace expansion along the column of
1's, which equals 0. The same argument proves the second identity:
in the common denominator we now have $V_q(a_1,...,a_q)$
and the numerator is a Vandermonde determinant of size
$(q-1)\times (q-1)$ minus its Laplace expansion along the column
of 1's . In the sum of $\frac{a_i^k}{Q'(a_i)}, \quad k
\leq q-1$, the numerator is the Laplace
expansion of \[\vmatrix 1& a_1&...a_1^{q-1}& 1\\
                       \hdotsfor4\\
                       1& a_n&...a_q^{q-1}& 1\\
                       \endvmatrix\]

along the column containing terms of the type $a^k$, and in
the sum of $\frac{a_i^{q+l}}{Q'(a_i)}$ the numerator
is the Laplace expansion of $V_{q,l}(a_1,...a_q)$ with respect to the last
column, which is a product of $V_q(a_1,...a_q)$ by the complete homogeneous symmetric
polynomial $S_l(a_1,...a_q)$ of degree $l>0$.

\end{proof}

\begin{rk}
The last identity was obtained in a different way as Theorem 3.2 in \cite{Co}, where it is also traced back to C.G.J. Jacobi.
\end{rk}
\begin{proof} (of Theorem \ref{thm:main})

First note that for $a \in K$ we can identify $(z-a)^{-1}$ with $\sum_{n=1}^{\infty}a^n z^{-(n+1)}$. More generally, if $Q(z)=z(z-a_1)...(z-a_q)$ is a square-free polynomial with $a_n \in K$ and $P \in K[z]$, then partial fraction decomposition gives $P/Q= c_0/z +c_1/(z-a_1)+...+c_q/(z-a_q)$ with $c_n=P(a_n)/Q'(a_n), \quad n=1,...,q$ (cf. \cite{Tr}) and $P/Q$ can also be identified with an element of $K[[1/z]]$. It follows that $\log (1-a/z)$ can be identified with $\sum_{n=1}^{\infty}((-1)^na^n/n)z^{-n}$. Let now $g=\int (1/Q)$ with $Q$ as above. Then $g=b_0+\frac{1}{Q'(0)}\log z +\frac{1}{Q'(a_1)}\log (z-a_1)+...+\frac{1}{Q'(a_q)}\log (z-a_q)$ in $\mathcal{F}$ with $b_0$ in $K$. Identifying each $\log (z-a_j), \quad j=1,...,q$ with an appropriate formal series in $K[[1/z]]$ as above and adding the results, we get $g=\sum_{n=0}^{\infty}b_nz^{-n}$. By the lemma, $b_1=...=b_{q-1}=0$, $b_q=1/q$ and $b_{q+l}=S_l(a_1,...,a_q)/(q+1)$, where $S_l$ is the complete homogeneous symmetric polynomial of degree $l$, $l=1,2,...$. To ensure this is the only possible series in $K[[1/z]]$ that can be identified with $g=\int (1/Q)$, note that any such series should be symmetric with respect to $a_1,...,a_q$. Theorem 9.3 in \cite{LS} says the following: If a formal series $F$ in the variables $X_1,...,X_q,Y$ is symmetric with respect to $(X_1,...,X_q)$, then $F=\Phi(\sigma_1,...,\sigma_q,Y)$, where $\Phi$ is a formal series in the variables $X_1,...,X_q,Y$ and  $\sigma_1,...,\sigma_q$ are elementary symmetric polynomials in $X_1,...,X_q$. Moreover, the series $\Phi$ is unique. Uniqueness of our $\sum_{n=0}^{\infty}b_nz^{-n}$ follows, because $S_l(X_1,...,X_q)$ can be expressed as polynomials in $\sigma_1,...,\sigma_q$. Hence also $o(g)=q$.
\end{proof}

\section{Applications}\label{s:apps}

A particular case of our Theorem \ref{thm:main} is Proposition 1 in \cite{GS}, which was proved for $K=\mathbb{C}$ and $a_j=\zeta^j a_1, j=1,...,q$, where $\zeta$ is a primitive root of unity of order $q$. The convergence of $\int 1/(z(z-a_1)...(z-a_q)) \to -1/(qz^q)$ as $a_1,...,a_q \to 0$ (which is in fact uniform for $|z| >R$) is important in constructing approximate Fatou coordinates
for analytic maps $f$ in a neighborhood of an 
$f_0(z)= z+z^{q+1}+...$ with $q>1$. These are coordinates in which $f$ looks like a translation. The first step in constructing Fatou coordinate for $f_0$ consists in lifting $f_0$ to a neighborhood of infinity by the coordinate change $z \mapsto -1/(qz^q)$. We considered $f$ belonging to an  one-parameter family of polynomials
$P_\lambda(z)=\lambda z+z^2$ with $\lambda_0=e^{2\pi ip/q}$ and
$\lambda=e^{2\pi i(p/q+u)}$, with $p,q$ coprime integers and $u$
in a sufficiently small neighborhood of $0$ in $\mathbb{C}$. We started the construction of near-Fatou coordinate by applying  
the transformation $w(z) =\int_{z_0}^z (1/Q(u,\zeta))d\zeta$, where $Q(u,z)$ is the Weierstrass polynomial for $P_{\lambda}^{\circ q}(z)-z$. As $u$ is small, the non-zero solutions of $Q(u,z)=0$ are also small. Because of convergence of integrals, the coordinates obtained for $P_\lambda$ depend continuously on $u$. For more details and references see \cite{GS}.\\

Another application is a generalization of the well-known formula for electrostatic potential of a dipole located at $z=0$ (cf. \cite{Ne}):
consider two charges $1/a$ and $-1/a$ placed respectively at
$z=0$ and $z=a$. Then their combined electrostatic potential is $(1/a)\log z -(1/a)\log(z-a)$, which tends to $-1/z$ as $a \to 0$ (this follows easily from the definition of derivative of the logarithmic function). In the setting of this paper, there are charges $1/Q'(0), 1/Q'(a_1),...,1/Q'(a_q)$ placed at $0,a_1,...,a_q$ (with $Q(z)=z(z-a_1)...(z-a_q)$), and it follows from Theorem \ref{thm:main} that the potential tends to $-1/(qz^q)$ as $a_1,...,a_q \to 0$.\\

\textbf{Acknowledgment:} The main idea of the paper occured while I was working with Estela Gavosto on \cite{GS}. I thank her for many useful conversations. 

\end{document}